\documentclass[11 pt]{amsart}
\usepackage{amsmath, amsthm,amssymb,bbm,enumerate,mathrsfs,mathtools}
\usepackage{enumitem,a4wide}
\usepackage[margin=3cm]{geometry} 
\usepackage{tikz,verbatim,xcolor,graphicx}
\usepackage[linktocpage=true]{hyperref}
\usepackage{makecell}
\usepackage{verbatim,bbm,theoremref}
\usepackage{marginnote,caption,subcaption}

\renewcommand\labelenumi{(\roman{enumi})}
\renewcommand\theenumi\labelenumi
\newcommand{\eps}{\ensuremath{\varepsilon}}

\def\F{\mathbb{F}}
\def\C{\mathbb{C}}
\def\Z{\mathbb{Z}}

\def\E{\ensuremath{\mathbb{E}}}     
\def\Ex{\ensuremath{\boldsymbol{E}}}    

\newcommand{\se}{\ensuremath{\subseteq}}
\newcommand{\sm}{\ensuremath{\setminus}}
\DeclareMathOperator{\supp}{supp}
\DeclareMathOperator{\Var}{Var}
\DeclareMathOperator{\spa}{span}
\DeclareMathOperator{\codim}{codim}
\DeclareMathOperator{\sol}{sol}

\title{On the uncommonness of minimal rank-2 systems of linear equations}

\author{Daniel Altman}
\address{Department of Mathematics, University of Michigan, Ann Arbor, U.S.A.}
\email{daniel.h.altman@gmail.com} 

\author{Anita Liebenau}
\address{School of Mathematics and Statistics, UNSW Sydney, NSW 2052, Australia.}
\email{a.liebenau@unsw.edu.au}

\begin{document}


\newtheorem{thm}{Theorem}
\numberwithin{thm}{section} 
\newtheorem{lem}[thm]{Lemma}
\newtheorem{prop}[thm]{Proposition}
\newtheorem{obs}[thm]{Observation}
\newtheorem{conj}[thm]{Conjecture}
\newtheorem{cor}[thm]{Corollary}
\newtheorem{claim}[thm]{Claim}
\newtheorem{ques}[thm]{Question}

\maketitle

\begin{abstract}
    We prove that suitably generic pairs of linear equations on an even number of variables are uncommon. This verifies a conjecture of Kam\v{c}ev, Morrison and the second author. Moreover, we prove that any large system containing such a $(2\times k)$-system as a minimal subsystem is uncommon. 
\end{abstract}

\section{Introduction}

Sidorenko's conjecture~\cite{sid93} is one of the biggest open problems in extremal graph theory. Despite much attention~\cite{sid1,sid2,sid3,sid4,sid5,sid6,sid7}, the conjecture remains wide open. The related concept of {\em commonness} has equally received a lot of attention. A graph $F$ is called {\em common} if the number of monochromatic copies of $F$ in $K_n$  is asymptotically minimised by a random colouring. In 1959, Goodman~\cite{goodman59} showed that the triangle is common. This led Erd\H{o}s~\cite{erdos62} to conjecture that $K_4$ is common.  Later, Burr and Rosta~\cite{br80} extended Erd\H{o}s' conjecture to arbitrary graphs. In 1989, the Burr-Rosta conjecture was refuted independently by Sidorenko~\cite{sid89} and by  Thomason~\cite{thomason89} who also refuted Erd\H{o}s' conjecture. In fact, Jagger, Stovicek and Thomason~\cite{jst96} later proved that every graph containing a copy of $K_4$ is uncommon, certifying the falsity of the aforementioned conjectures in a rather strong sense. 

More recently, the arithmetic analogue has received considerable attention. Saad and Wolf~\cite{sw17} initialised the systematic study of common and Sidorenko systems of linear equations. A linear pattern $\Psi = (\psi_1,\ldots, \psi_t)$ which maps $\Z^D \to \Z^t$ is said to be \textit{Sidorenko} if 
\[ T_\Psi(f) := \E_{x_1,\ldots,x_D \in \F_p^n} f(\psi_1(x_1,\ldots,x_D))\cdots f(\psi_t(x_1,\ldots, x_D)) \geq \alpha^t,\]
for all $\alpha \in [0,1]$, for all $n$ and for all $f: \F_p^n \to [0,1]$ with $\E_x f(x) = \alpha$. Here and throughout, we use the $\E$ notation to denote a normalised sum (i.e. a sum normalised by the number of elements in the set over which we sum). That is, the Sidorenko property asks that the number of solutions to $\Psi$ weighted by any $f$ of average $\alpha$ is at least that of the constant function $\alpha$. By sampling a set $A$ randomly according to the function $f$, the Sidorenko property can be viewed as asking that \textit{any} set $A$ contains at least as many copies of the pattern $\Psi$ as what is expected from a random set of the same density. 

The commonness property is similar, but for 2-colourings. A linear pattern $\Psi$ is said to be common if 
\[T_\Psi(f) + T_\Psi(1-f) \geq 2^{1-t},\]
for all $f:\F_p^n \to [0,1]$. Colloquially, it asks that \textit{any} two colouring of $\F_p^n$ contains at least as many monochomatic patterns $\Psi$ as what is expected of a random two colouring. 

While both the Sidorenko property and commonness property for single equation systems (i.e., when the image of $\Psi$ has codimension 1) are fully characterised~\cite{fpz19,sw17}, both general classification problems (i.e., for systems with codimension $> 1$) remain wide open. 

Saad and Wolf showed that the linear pattern corresponding to a four-term arithmetic progression (i.e. $\Psi(x,y) = (x,x+y,x+2y,x+3y)$) is uncommon over $\F_5$. They asked whether every system containing a 4-AP is uncommon, a question that was answered in the affirmative by Kam\v{c}ev, Morrison and the second author~\cite{KLM22}, and independently by Versteegen~\cite{v21}. Here, we say that a $k$-variable system $\Psi$ contains an $\ell$-variable system $\Psi'$ if there exists a subset $B$ of $\ell$ variables such that every solution to $\Psi$ forces the variables in $B$ to be a solution to $\Psi'$.\footnote{Here and throughout we abuse notation and terminology to refer to $\Psi$ interchangeably as both a system of linear forms, and from the dual perspective as a system of linear equations. In this way, we may refer to ``solutions of $\Psi$''.} We give a formal definition in Section~\ref{s:prelims}. 
The result of~\cite{KLM22} proves the more general statement in which a 4-AP may be replaced by any suitably nondegenerate system with $D=2, t=4$. 

In what follows, we may refer to a system $\Psi$ with parameters $D, t$ as a $(t-D)\times t$ system, so  for example a 4-AP is a $2\times 4$ system. We denote by $s(\Psi)$ the length of the shortest equation contained in $\Psi,$ where the {\em length} of an equation is the number of non-zero coefficients. 
Kam\v{c}ev, Morrison and the second author also conjectured the following.

\begin{conj}[Conjecture 6.1 in~\cite{KLM22}]\thlabel{conj:KLM2xk}
    For even $k\ge 6$, any $(2\times k)$-system $\Psi$ with $s(\Psi) = k-1$ is uncommon. 
\end{conj}
As mentioned above, the case $k=4$ is proved in~\cite{KLM22}. The assumption for $k$ to be even is necessary, since some such systems with $k$ odd are known to be common, see~\cite{KLM23} for details. These $(2\times k)$-systems with $s(\Psi)=k-1$ are important since, by localising (see \cite{A22-localSid}), they can be used to prove uncommonness of larger systems containing them. We provide the details in Section~\ref{s:prelims}. 

In this paper, we prove \thref{conj:KLM2xk}. Moreover, using methods that the first author developed in~\cite{A22-localSid}, we are able to generalise~\cite{KLM22} to larger systems. 

\begin{thm}\thlabel{thm:main}
Let $p$ be a large enough prime, 
let $k \geq 4$ be even. Let $\Psi$ be a $2\times k$ system such that $s(\Psi)=k-1.$
Let $m\ge 2$, $t\geq k$. If $\Psi'$ is an $m\times t$ system that has $s(\Psi') = k-1$ and that contains $\Psi$, then $\Psi'$ is uncommon. 
\end{thm}

By taking $\Psi' = \Psi,$ we see that in particular that \thref{conj:KLM2xk} is proved. We remark that an independent proof of \thref{thm:main} has been announced by Dong, Li and Zhao \cite{DLZ24}. \newline

{\bf{Acknowledgments.}} We thank Sarah Peluse for valuable feedback on an earlier version of this paper. The second author is partially supported by the Australian Research Council DP220103074.

\section{Preliminaries}\label{s:prelims}

Throughout this document we assume that $p$ is an odd prime, and all asymptotic notation will be implicitly with respect to the limit $n\to \infty$.

Let $\Psi$ be a system of $m$ linear equations in $t$ variables, described by an $m\times t$ matrix $M_{\Psi}$ with entries in $\F_p$. 
For $n\ge 1,$ we denote by $\sol(\Psi;\F_p^n)$ the set of $x\in (\F_p^n)^t$ that are solutions to the homogeneous system of equations $M_\Psi x =0.$ 
For a function $f : \F_p^n \to \C$ recall the solution density $T_{\Psi}(f)$ and note that this may equivalently be defined by
$$T_{\Psi}(f) = \E_{x\in\sol(\Psi;\F_p^n)} f(x_1)f(x_2) \cdots f(x_t),$$ 
by adapting the dual view on the solution space. 
We always assume that $t>m$ and that the system is of full rank. In particular, $\Psi$ has nontrivial solutions. 
We use the Fourier transform on $\F_p^n$. Let $e_p(\cdot)$ be shorthand for $e^{2\pi i/p}$. For $h\in F_p^n$, we
define
$$\hat{f}(h) := \E_{x\in\F_p^n} f(x) e_p(x\cdot h).$$ 
We shall use the following identity which can be readily obtained by Fourier inversion. For any $f:\F_p^n \to [-1,1]$  we have 
\begin{equation}\label{e:tpsi-fi} T_\Psi(f) = \sum_{r,s \in \F_p^n} \prod_{i=1}^k \hat f(a_i r + b_i s).
\end{equation}

A multiset $W\se \F_p^n$ is said {\em to form a cancelling partition} if the elements of $W$ be indexed as $W=\{v_1,\ldots,v_{|W|/2},w_1,\ldots,w_{|W|/2}\}$ such that $v_i+w_i=0$ for all $1\le i\le |W|/2.$ 

Finally, define the tensor product of $f_1:\F_p^{n_1} \to \C$ and $f_2:\F_p^{n_2}\to \C$ as follows: let $f_1\otimes f_2:\F_p^{n_1 + n_2}\to \C$ be given by $f((x_1,x_2)) = f_1(x_1)f_2(x_2)$, where $x_i \in \F_p^{n_i}$ and we concatenate vectors in the obvious way. It is easily checked that the tensor product is multiplicative in the following sense. For any $\Psi$, we have  $T_\Psi(f_1\otimes f_2)=T_\Psi(f_1)T_\Psi(f_2).$

\section{Main section}

We prove \thref{thm:main} in two steps. Firstly, we prove that for every 
$(2\times k)$-system $\Psi$ with $s(\Psi)=k-1$ there is a function $f_{\Psi}$ with $\E f_{\Psi}=0$ such that $T_{\Psi}(f_{\Psi}) < 0.$ This already implies that any such $\Psi$ is uncommon. We then use the idea of tensoring to find a function $f$ that 
witnesses uncommonness of a system $\Psi'$ containing such $\Psi$ as a minimal subsystem.

We first give a short proof of the fact that a $2\times k$ system $\Psi$ with $s(\Psi)=k-1$ is uncommon under the additional assumption that $\Psi$ does not contain an additive $k$-tuple.

\begin{thm}\thlabel{thm:noAddKTuple}
Let $p>2$ be a prime and let $\Psi$ be a $2\times k$ system such that $s(\Psi)=k-1$ 
and which does not contain an additive $k$-tuple. Then there exists $f:\F_p^2 \to [-1,1]$ with $\E f = 0$ and such that $T_\Psi(f) < 0$.
\end{thm}

In the proof of \thref{thm:noAddKTuple} and the more general version below, we will define a random function $F$. For a random variable $X$ defined by $F,$ we denote by $\Ex_{F} X_F$ the expectation of $X$ with respect to this random function. 

\begin{proof}[Proof of \thref{thm:noAddKTuple}]
We first note that the condition $s(\Psi)=k-1$ is equivalent to the fact that every $2\times 2$ minor of $M_\Psi$ has full rank.  
Let $a_1,\ldots,a_k,b_1,\ldots,b_k\in\F_p$ such that \[M_\Psi := \begin{pmatrix} a_1 & a_2 & \cdots & a_k \\
b_1 & b_2 & \cdots & b_k\end{pmatrix}.\] 
Recall by Fourier inversion that for any $f:\F_p^2 \to [-1,1]$ 
\begin{equation}\label{e:tpsi-fi} T_\Psi(f) = \sum_{r,s \in \F_p^2} \prod_{i=1}^k \hat f(a_i r + b_i s).
\end{equation}
For $r,s \in \F_p^2$, define the multiset $M(r,s):= \{a_ir + b_is : i \in [k] \}$. Let $r_0, s_0$ be any two linearly independent vectors in $\F_p^2$ and let $M_0 := M(r_0,s_0)$. It is easily checked using the linear independence of $r_0$ and $s_0$ and the rank condition of $M_\Psi$ that $0\not \in M_0$, that $|M_0| =k$ and that $M_0 \cap -M_0 = \emptyset$. 

We will identify an $f$ as stated in the theorem by sampling functions $F$ randomly and showing that with positive probability, there exists an instance which satisfies the desired conditions. In fact, we will specify the Fourier transform $\hat F$ of $F$. We will choose $\hat F$ to be supported on $M_0 \cup -M_0$. For every $h \in M_0$, let $\hat F(h) = {\xi_h}/{2k}$, where $\xi_h$ is chosen uniformly at random (and independently from previous choices) on $S^1\subset \C$, and let $\hat F(-h) := \overline{\hat F(h)}$. Let $\hat F(h)$ be zero elsewhere. By our observations above, $\hat F$ and hence $F$ is well-defined. Also note by Fourier inversion and the triangle inequality that $\|F\|_\infty \leq 1$ and furthermore that $F$ is real-valued, so indeed $F$ takes values in $[-1,1]$. Furthermore, $0 \not \in \supp{\hat F}$ and so $\E F = 0$.

\begin{claim}
For $r,s\in\F_p^2,$ if $M(r,s)\se M_0\cup -M_0$ then $M(r,s)$ does not form a cancelling partition. 
\end{claim}
\begin{proof}
If $r,s$ are linearly independent, we observe that if $a_ir+b_is = -a_j r -b_js$ then $a_i=-a_j$ and $b_i=-b_j$ which is impossible by the condition on the ranks of the $2\times 2$ minors of $M_\Psi$ and the fact that $p\ne 2$.

Next suppose that $s = \lambda r$ for some $\lambda \in \F_p$, so that $\{a_ir +  b_is\}_{i=1}^k$ lies in a one-dimensional subspace. Furthermore, if $a_ir_0 + b_is_0 = c\cdot \pm (a_j r_0 + b_j s_0)$ for some $c\in \F_p$ then by the linear independence of $r_0,s_0$ and the rank condition we see that $c=\pm 1$ and $i=j$; that is, $M_0\cup -M_0$ contains at most two points (which must sum to zero) from any one dimensional subspace of $\F_p^2$. Therefore, $M(r,s)$ must be contained in $\{\pm (a_ir_0 + b_is_0) \}$ for some particular $i$. Therefore, if $M(r,s) = \{(a_i + \lambda b_i)r\}_{i=1}^k$ forms a cancelling partition, it must do so on only two values, and so by adding $\lambda$ times the second row of $M_\Psi$ to the first, we obtain a scalar multiple of an additive $k$-tuple, a contradiction to the assumption on $M_\Psi$.
\end{proof}

By definition, $\supp(\hat F) = M_0\cup -M_0,$ and so continuing from (\ref{e:tpsi-fi}), we see
\[ T_\Psi(F) = \sum_{r,s:M(r,s)\subseteq M_0 \cup -M_0}\prod_{i=1}^k \hat F(a_i r + b_i s).\] 
By the previous claim, if $M(r,s)\subseteq M_0 \cup -M_0$ then it does not form a cancelling partition. Thus, there is an element $h\in M(r,s)$ such that $\hat F(h)$ is chosen uniformly and independent of $M(r,s)\sm\{h\}$. It follows that every product in the sum is uniformly distributed on $S^1$. 

Thus, we have

\[ \Ex_F T_\Psi(F) = \sum_{r,s:M(r,s)\subseteq M_0 \cup -M_0}\Ex_F\prod_{i=1}^k \hat F(a_i r + b_i s)=0.\]

Let $\eps>0$. With positive probability depending on $\eps$, $|\xi_h-1|<\eps$ for all $h \in M_0\cup -M_0$ and so $\Re \left( \prod_{i=1}^k \hat F(a_i r + b_i s) \right)\geq 1/2$ \textit{simultaneously} for all $r,s$ such that $M(r,s) \subseteq M_0 \cup-M_0$ (note that at least one pair $(r,s)=(r_0,s_0)$ exists so that $M(r,s) \subseteq M_0 \cup-M_0$). Thus, with positive probability, 
\[ T_\Psi(F) = \Re(T_\Psi(F)) = \sum_{r,s:M(r,s)\subseteq M_0 \cup -M_0}\Re\left( \prod_{i=1}^k \hat F(a_i r + b_i s)\right) >0.\]
Therefore, since $\Ex_F T_\Psi(F) = 0$, there must exist a particular instance $f$ such that $T_\Psi(f)<0$, completing the proof.
\end{proof}

We next prove the following which is the main novelty of this paper. 

\begin{thm}\thlabel{thm:main-2xk-uncommon}
Let $p>2$ be prime and let $k \geq 3$. Let $\Psi$ be a $2\times k$ system such that $s(\Psi)=k-1$.
Then there exists an integer $n >0$ and a function $f: \F_p^n \to [-1,1]$ with $\E f = 0$ such that $T_\Psi(f) < 0$.
\end{thm}

In the proof we make use, apart from the generalised Fourier inversion formula, of a central limit theorem by McLeish for martingale difference arrays. An array $\{Y_{n,i} \mid i =1,\ldots, k_n\}$ of discrete random variables is called a {\em martingale difference array} if for all $n$ and $i$, $\E (Y_{n,i} \mid \{ Y_{n,j} : 1\le j <i\} ) =0.$ 

\begin{thm}[{McLeish~\cite[Corollary 2.13]{M74}}]\thlabel{mcleish}
    Let $\{Y_{n,i}\}$ be a martingale difference array normalised by its variance such that 
    \begin{itemize}
    \item[(M1)] $\lim_{n\to \infty} \sum_{i=1}^{k_n} \Ex\left(Y_{n,i}^2 1_{|Y_{n,i}|> \varepsilon} \right) = 0$ for every $\varepsilon > 0$, and
    \item[(M2)] $\limsup_{n \to \infty} \sum_{i\ne j} \Ex \left( Y_{n,i}^2Y_{n,j}^2 \right) \leq 1$.
\end{itemize} Then the sum $S_n := \sum_{i=1}^{k_n} Y_{n,i}$ converges in distribution to $\mathcal{N}(0,1)$.
\end{thm}
 We remark that {\em normalised by variance} here means (the obvious) that the random variables are normalised so that the variance of $\sum_{i=1}^{k_n} Y_{n,i}$ tends to 1 as $n\to \infty.$

\begin{proof}[Proof of \thref{thm:main-2xk-uncommon}]
We start with a brief note that this theorem is straightforward to prove when $k$ is odd. Indeed, it is easily checked that there exists a function $f : \F_p^n \to [-1,1]$ with $\E f = 0$ such that $T_\Psi(f) \neq 0$, by using the assumption $s(\Psi)=k-1$, or, equivalently, the fact that all $2\times 2$ minors of $M_{\Psi}$ have full rank. If $T(f)>0$ then $T(-f)<0,$ using that the number of variables is odd. We assume from now that $k$ is even. 

As we did in the previous theorem, we determine $f$ by specifying its Fourier coefficients. Most of them will be chosen randomly and we will let $F$ be this random function. For the rest of this proof, it will be convenient to have a total ordering on $\F_p^n$, and a notion of positive and negative elements. Exactly how this is done is not so important, but we will proceed as follows for concreteness. We identify each element of $\F_p^n$ with an element of $[ -(p-1)/2, (p-1)/2]^n$ in the natural way. Each coordinate then has an ordering inherited from the integers, and we will imbue $\F_p^n$ with the corresponding lexicographical ordering; we will denote it by $\preceq$. For $h \in \F_p^n$, we will also define the absolute value notation 
\[ |h| :=\begin{cases}
      h & \text{if }  h  \succeq 0\\
      -h & \text{if } h \prec 0
\end{cases}.\]
We note that $h  \succeq 0$ if and only if the first non-zero coordinate of $h$ is $\ge 0$ (as a natural number). 
For each $h \succ 0$, let $\hat F(h) := \xi_h$, where each $\xi_h$ is chosen uniformly at random (and independently from previous choices) on $S^1 \subset \C$, and let $\hat F(-h):= \overline{\hat F(h)}$. Let $\hat F(0)=0$ deterministically.
Similarly to the proof of \thref{thm:noAddKTuple} we note that $\| F\|_\infty \leq p^n$, by Fourier inversion and the triangle inequality, and that $F$ is real-valued, so indeed $F$ takes values in $[-p^n,p^n]$. 

Let us define 
\begin{align}
    P(\hat F) &= \sum_{r,s \text{ l.i.}} \prod_{i=1}^k \hat F(a_i r + b_i s), \\
    \text{and } Q(\hat  F) &= \sum_{r,s \text{ not l.i.}} \prod_{i=1}^k \hat F(a_i r + b_i s),\nonumber 
\end{align}
where the sums are over $r,s\in\F_p^n$ and ``l.i.'' stands for linearly independent. We note immediately that $|Q(\hat F)|\le p^{n+1}+p^n-p =O(p^n)$, where we recall that all asymptotic notation is  written with respect to $n\to \infty$. 

Then, by Fourier inversion, we have that 
\begin{equation}\label{e:scnd-tpsi-fi} T_\Psi(F) = \sum_{r,s \in \F_p^n} \prod_{i=1}^k \hat F(a_i r + b_i s) =  P(\hat F) + Q(\hat  F).
\end{equation}

\begin{lem}\thlabel{aux:P(F)tendsToNormal} For the random function $F$ we have 
$\Ex_FP(\hat F) = 0$, $\Var P(\hat F) =\Omega( p^{2n})$, and 
$P(\hat F)/\sqrt{\Var P(\hat F)}\xrightarrow{d} \mathcal{N}(0,1)$. 
\end{lem}
Let us first argue how the lemma implies the theorem. Since $Q(\hat F)=O(p^n),$ there is a constant $c>0$ such that $Q(\hat F)/\sqrt{\Var P(\hat F)}< c$ for $n$ large enough. 
Then, denoting by $\Phi$ the CDF of the standard normal distribution, we have 
\[\mathbb{P}(T_\Psi(F) < 0) \geq \mathbb{P} \left( \frac{P(\hat F)}{\sqrt{\Var P(\hat F)}} < -c\right) > \Phi(-c)/2 > 0\]
for $n$ sufficiently large. For this $n$, we then let $f_0:\F_p^n \to [-p^n, p^n]$ be such a function with $T_\Psi(f_0) < 0$. The theorem follows by setting $f= f_0/p^n$.

\begin{proof}[Proof of \thref{aux:P(F)tendsToNormal}]
    First note that, for $r,s$ linearly independent and by the condition on the $2\times 2$ minors of $M_\Psi$, that $|a_ir + b_i s|$ are distinct for $i \in [k]$. Thus we have that $\Ex_FP(\hat F) = 0$. 

Let us now set up a martingale difference array. 
As in the previous proof, we define the multiset $M(r,s) := \{ a_ir + b_is : i \in [k]\}$ for $r,s\in\F_p^n$. We will also define $$|M(r,s)| := \{ |m| : m \in M(r,s)\}.$$ As noted above, the entries of $M(r,s)$ and of $|M(r,s)|$ are in fact distinct when $r$ and $s$ are linearly independent, and furthermore $0 \not \in M(r,s)$ for any linearly independent $(r,s)$. 
Next, for $h\in \F_p^n$ with $h\succeq 0$, define the multiset \[A_h := \{M(r,s): (r,s) \text{ l.i.}, h \in M(r,s), \text{ and } h \succeq |h'| \text{ for all } h' \in M(r,s)\},\]
i.e.~those terms in $P(\hat F)$ for which $h$ is the largest entry. Finally, let
\[ X_h := 2\sum_{M\in A_h} \Re \left( \prod_{m \in M}  \hat F(m)\right) = \sum_{M\in A_h} \left(\prod_{m \in M} \hat F(m) + \prod_{m \in -M} \hat F(m) \right),\]
so that $P(\hat F) = \sum_{h \succ 0} X_h$, where we note that $h=0$ does not appear in the sum due to the linearly independent assumption of $r$ and $s$. 

We note that $\Ex(\hat F(h)| \{X_{h'} : 0 \prec h' \prec h\}) = \Ex(\hat F(h)) = 0$, by independence. Thus, 
$\Ex\left(X_h| \{X_{h'} : 0 \prec h' \prec h\}\right) = 0$ for every $n,$ that is, the collection $\{X_h = X_{n,h}\}$ is indeed a martingale difference array. 

Next, we have 
\[ \Var P(\hat F) = \sum_{h \succ 0} \Ex(X_h^2) = 2\sum_{h \succ 0} \sum_{M,M'\in A_h} \left(\Re \Ex  \prod_{m \in M \cup M'}\hat F(m) + \Re \Ex \prod_{m \in M \cup -M'}\hat F(m) \right),\]
where the unions are taken as multisets. Each expectation in the above is 1 if $M \cup \pm M'$ forms a cancelling partition, and 0 otherwise. Therefore, by restricting to the subsum when $M=M'$ and noting that the double sum $\sum_{h\succ 0}\sum_{M\in A_h}$ then simplifies to $\sum_{M: r,s \text{ l.i.}},$ we see that $\Var P(\hat F) = \Omega( p^{2n})$.
Let $V := \Var P(\hat F)$, let $\tilde X_h  = V^{-1/2} X_h$ so $\{\tilde X_h\}$ is a martingale difference array normalised to have variance 1, and let $\tilde P(\hat F) = V^{-1/2} P(\hat F)$. It remains to verify the conditions (M1) and (M2) in \thref{mcleish}. 

We need the following linear algebra observation in both parts. 
\begin{obs}\thlabel{c:second}
Let $(r,s)$ be linearly independent and $(r',s')$ be linearly independent. If $|M(r,s)| \cap |M(r',s')| = 1$ then $\dim \spa(r,s,r',s') \leq 3$, and if $|M(r,s)| \cap |M(r',s')| \geq 2$ then $\dim \spa(r,s,r',s') = 2$.
\end{obs}
\begin{proof}
     Here we observe that for fixed $r,s$, the map $(r,s) \mapsto |M(r,s)|$ is (prior to removing the ordering of the output) given by a $2\times k$ matrix (say, $\widetilde{M}$) obtained by negating certain columns of $M_\Psi$. In particular, $\widetilde{M}$ has all of its $2\times 2$ minors of full rank. The same is of course true of the matrix obtained from $|M(r',s')|$. The claim then follows from this condition on the $2\times 2$ minors of these matrices.
\end{proof}

\medskip

\noindent
{\em Proof of (M1).}
Observe that $\sum_{h \succ 0} \Ex\left(\tilde X_h^2 1_{|\tilde X_h|> \varepsilon} \right) \leq \frac{1}{\varepsilon^2} \sum_{h \succ 0} \Ex(\tilde X_h^4)$, so it suffices to show that $\sum_{h \succ 0}\Ex(\tilde X_h^4) \to 0$ as $n\to \infty$. To this end we compute 
\[ \sum_{h \succ 0} \Ex(X_h^4) = \sum_{h \succ 0} \sum_{M_1,M_2,M_3,M_4\in A_h} \sum_{\rho \in \{-1,1\}^4} \Ex\left(  \prod_{m \in \rho(M_1\cup M_2\cup M_3\cup M_4)} \hat F(m)\right),\]
where, for example, for $\rho = (1,-1,-1,1)$ we have $\rho(M_1\cup M_2 \cup M_3 \cup M_4) = M_1 \cup -M_2 \cup -M_3 \cup M_4$. As before, each of these expectations is zero unless $\rho(\cup_{i=1}^4 M_i)$ forms a cancelling partition. This in turn implies that the elements of $\cup_{i=1}^4 |M_i|$ coincide in pairs. 
Thus, to show (M1), it suffices for us to show that 
$$ T:= \# \{(h,M_1,M_2,M_3,M_4)\mid h \succ 0, M_i \in A_h, \cup_{i=1}^4 |M_i| \text{ pairs}\} = o(p^{4n}).$$ 

Fix now $(h,M_1,M_2,M_3,M_4)\in T$, and, for each $i\in [4]$, fix linearly independent $(r_i,s_i)$ such that 
$M_i = M(r_i,s_i),$ and let $\widetilde{M}_i$ be the matrix such that $\widetilde{M}_i(r_i,s_i) = |M(r_i,s_i)|$. 
Consider an auxiliary $4$-partite graph on $4k$ vertices, where the vertices in the $i$th part are the elements of $\widetilde{M}_i(r_i,s_i)$ for $i=1,2,3,4$; and where $uv$ is an edge if and only if $u$ and $v$ are equal as elements of $\F_p^n$. We note that there is no edge whose endpoints lie in the same part, again since the elements of any one $\widetilde{M}_i(r_i,s_i)$ are distinct. An illustration can be found in Figure~\ref{fig:1}. 
\begin{figure}[b]
    \centering
        \includegraphics[page=4,width=.45\textwidth]{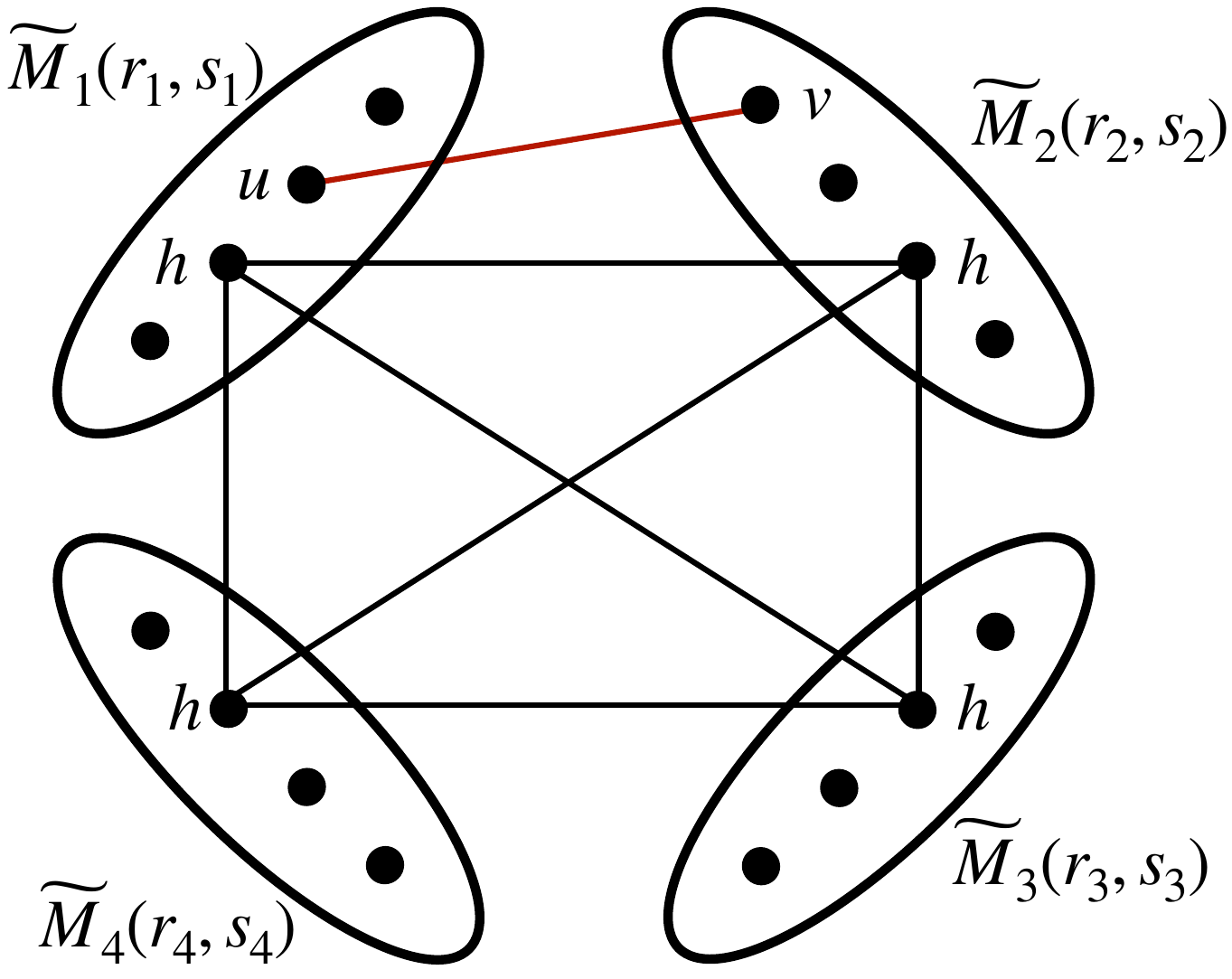}
        \caption{An illustration of the auxiliary graph used to prove (M1). An edge $uv$ is included if the two elements $u=a_ir_1+b_is_1$ and $v=a_jr_2+b_js_2$ are equal as elements of $\F_p^n$.}
        \label{fig:1}
\end{figure}
The condition of $(h,M_1,M_2,M_3,M_4)$ being an element of $T$ implies that this graph must have a perfect matching. We claim that the number of such tuples is $O(p^{3n})$. 

Suppose such a graph $G$ has a perfect matching. We know that $G$ contains a copy of $K_4$ since $h\in \widetilde{M}_i(r_i,s_i)$ for each $i$. This, together with the fact that each part is independent, and that there is a perfect matching implies that there is a partition of the four parts into pairs such that each pair has at least two edges between them; without loss of generality, say the pairs are $\{1,2\}$ and $\{3,4\}$.

Fix choices of $(r_1,s_1)$ in at most $p^{2n}$ ways. We have $h \in \widetilde{M}_1(r_1,s_1)$ so there are $k=O(1)$ choices for $h$, and by \thref{c:second}, there are at most $p^2 = O(1)$ choices for $(r_2,s_2)$ since both lie in the span of $(r_1,s_1)$. Since part 2 and part 3 share an edge (corresponding to $h$), there are $O(p^n)$ choices for $(r_3,s_3)$, and since there are at least two edges between part 3 and part 4, there are subsequently $O(1)$ choices for $(r_4,s_4)$. In total, there are $O(p^{3n})$ choices, which completes the argument for (M1).

\medskip

\noindent
{\em Proof of (M2).} We have to prove that $\limsup_n \sum_{\substack{h_1\ne h_2\\ h_1,h_2 \succ 0}} \Ex(\tilde X_{h_1}^2\tilde X_{h_2}^2) \leq 1$.   
We have that 
\[\sum_{\substack{h_1\ne h_2\\ h_1,h_2 \succ 0}} \Ex(X_{h_1}^2X_{h_2}^2) 
= \sum_{\substack{h_1\ne h_2\\ h_1,h_2 \succ 0}} \sum_{M_1,M_1' \in A_{h_1}}  \sum_{M_2,M_2' \in A_{h_2}} \sum_{\rho\in\{\pm1\}^4}
\Ex\left(  \prod_{m \in \rho(M_1\cup M_1'\cup M_2\cup M_2')} \hat F(m)\right).\]
If $(|M_1|\cup |M_1'|)\cap (|M_2| \cup |M_2'|) = \emptyset$, then by independence we have
\[\Ex\left(  \prod_{m \in \rho(M_1\cup M_1'\cup M_2\cup M_2')} \hat F(m)\right) 
= \Ex\left(  \prod_{m \in \rho(M_1\cup M_1')} \hat F(m)\right) \Ex\left(  \prod_{m \in \rho(M_2\cup M_2')} \hat F(m)\right).\]
Inserting extra (nonnegative) terms into the sum as appropriate, we then have that 
\[\sum_{\substack{h_1\ne h_2\\ h_1,h_2 \succ 0}} \Ex(X_{h_1}^2X_{h_2}^2) \leq \left(\sum_{h\succ 0} \Ex(X_h^2)\right)^2 + \sum_{\substack{h_1\ne h_2\\ h_1,h_2 \succ 0}} \sum_{M_1,M_1',M_2,M_2'}   \Ex\left(  \prod_{m \in \rho(M_1\cup M_1'\cup M_2\cup M_2')} \hat F(m)\right),\]
where the second sum is now over $M_1,M_1' \in A_{h_1}$ and $M_2,M_2'\in A_{h_2}$ such that $(|M_1|\cup |M_1'|)\cap (|M_2| \cup |M_2'|) \ne \emptyset$. Dividing the above inequality through by $V^2$, we see that to show (M2) and complete the proof, it suffices to show that the double sum above is of size $o(p^{4n})$.

\begin{figure}[b]
    \centering
    \begin{subfigure}[t]{0.45\textwidth}
    \centering
        \includegraphics[page=5,width=.98\textwidth]{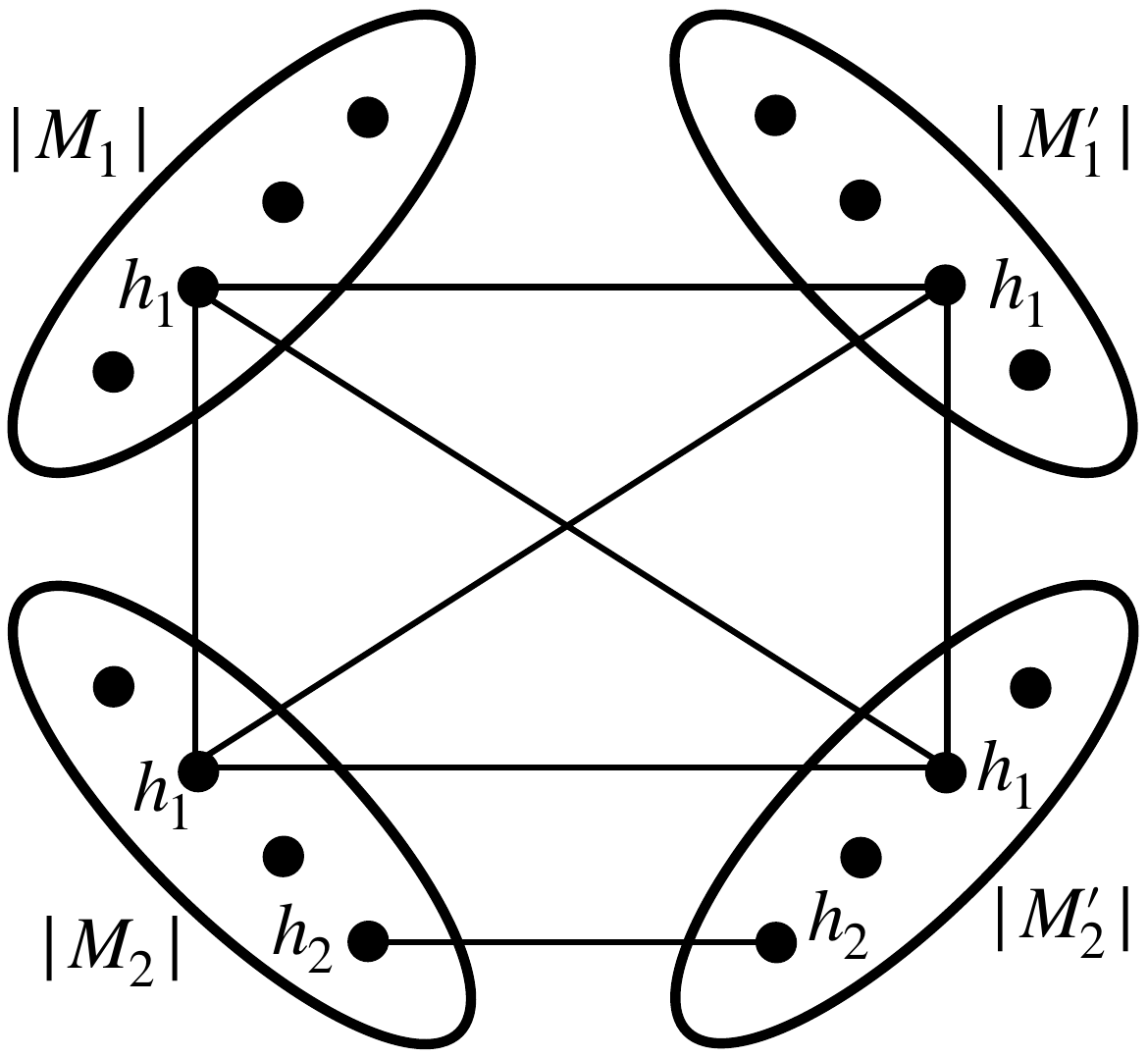}
        \caption{Figure 2}
        \label{fig:2}
    \end{subfigure}
    \hfill 
    \begin{subfigure}[t]{0.45\textwidth}
        \includegraphics[page=6,width=.96\textwidth]{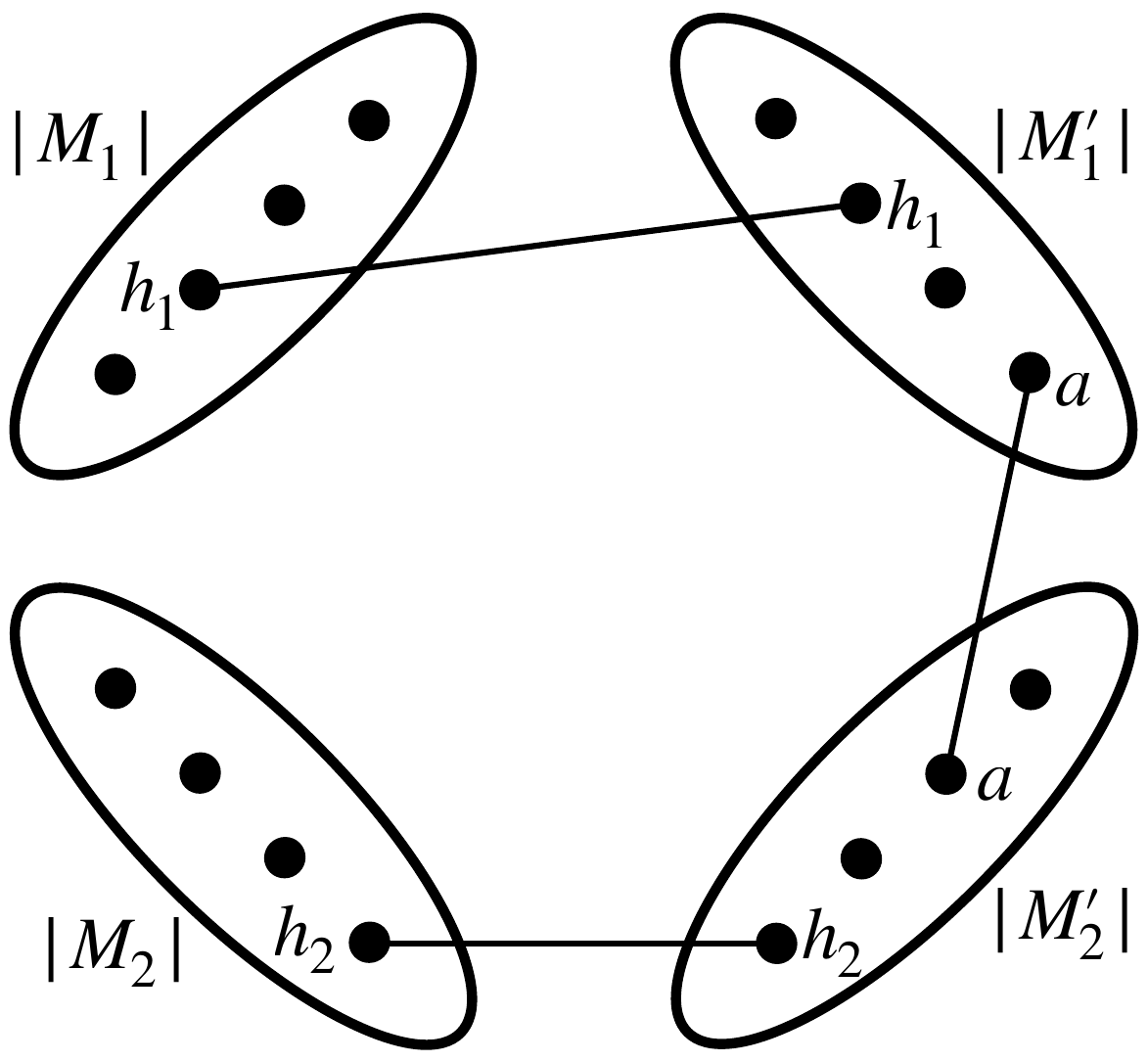}
        \caption{Figure 3}
        \label{fig:3}
    \end{subfigure}
\end{figure}
We now analyse the summands for which we get a nonzero contribution to the double sum. We have $h_1 \ne h_2$ and $0\prec h_1,h_2$; assume without loss of generality that $h_1 \prec h_2$. Recall that $h_1$ is maximal over elements of $|M_1|,|M_1'|$, and so we have that $h_2 \not \in |M_1|\cup |M_1'|$. Set up a $4$-partite graph as we did for the proof of (M1). We firstly deal with the (easier) case that $h_1 \in |M_2|\cup |M_2'|$. Note then that $h_1\in |M_2|\cap |M_2'|$ since it must occur an even number of times in the graph. An illustration is provided in Figure~2. 
First choose $h_1, h_2$ in at most $p^{2n}$ ways. There are then $O(1)$ choices for $(r_2,s_2,r_2',s_2')$. Choose an element of $|M_1|$ that is not $h_1$ in $O(p^n)$ ways. There are then $O(1)$ ways to complete $|M_1|$. By the pigeonhole principle, the part corresponding to $|M_1'|$ must have at least two edges in common with at least one other part, and so there are $O(1)$ ways to determine $|M_1'|$. This completes the proof in this case.

In the case that $h_1 \not \in |M_2|\cap |M_2'|$, we have some other element $a \in (|M_1|\cup |M_1'|)\cap (|M_2| \cup |M_2'|)$ where $a \ne h_1,h_2$. See Figure~3
for an illustration. There are at most $p^{2n}$ choices for $(r_1,s_1)$. By the pigeonhole principle again, there is another part $|M^{(2)}|$ which has at least two edges in common with that of $|M_1|$. We may then determine the $(r,s)$ which generate $M^{(2)}$ in $O(1)$ ways. Thanks to the existence of $a\ne h_1,h_2$ established above there is at least one edge between the parts of $|M_1|$ and $|M^{(2)}|$ and the remaining parts. Let $|M^{(3)}| \not \in \{|M_1|, |M^{(2)}|\}$ be the remaining part which is incident to this edge. Then there are $O(p^n)$ ways to determine $|M^{(3)}|$. Finally, there are $O(1)$ ways to determine the final part since, by the pigeonhole principle, it must have two edges in common with one of the previous parts. Thus there are $O(p^{3n})$ ways to obtain such a matching, so the contribution from the double sum above is $o(p^{4n})$, which completes the proof of (M2) and the proof of \thref{aux:P(F)tendsToNormal}. 
\end{proof}
As shown above, \thref{aux:P(F)tendsToNormal} implies the theorem. 
\end{proof}

We will need the following lemma to prove Theorem \ref{thm:main}.
\begin{lem}\label{l:one-bigger}
Let $\Psi_1, \ldots, \Psi_l$ be a sequence of linear patterns which are distinct in the sense that the operators $T_{\Psi_1}, \ldots, T_{\Psi_l}$ are distinct.\footnote{That is, the matrices corresponding to $\Psi_i$ are distinct up to row operations, and permutations of columns.} Then there exists $f:\F_p^n \to [-1,1]$ and some $i \in [l]$ such that $|T_{\Psi_i}(f)|>|T_{\Psi_j}(f)|$ for all $j\ne i$.
\end{lem}
\begin{proof}
View each $T_{\Psi_j}$ as a polynomial which maps $[-1,1]^{\F_p^n}$ to  $[-1,1]$. Since the operators $T_{\Psi_j}$ are distinct, the corresponding polynomials are distinct, and the equation $T_{\Psi_j}(f) = T_{\Psi_j'}(f)$ for $j\ne j'$ applies only to a measure 0 set of $f$. It suffices to find $f$ which avoids the union of all such equations; this set of $f$ has full measure.
\end{proof}

We are ready to prove our main theorem. 

\begin{proof}[Proof of \thref{thm:main}]
In this proof we will have two parameters $A,B \in \mathbb{N}^+$ whose values will be chosen later; we will first choose $A$ to be large depending on $\Psi, \Psi'$ and then choose $B$ to be large depending on $A, \Psi, \Psi'$. 

We wish to produce a function $f:\F_p^n \to [-1/2,1/2]$ with $\E f = 0$ such that 
\[T_{\Psi'}\left(\frac{1}{2} + f\right) + T_{\Psi'}\left(\frac{1}{2} - f\right) < 2^{1-t}.\]
By the usual multilinear expansion we have 
\[T_{\Psi'}(\frac{1}{2} + f) + T_{\Psi'}(\frac{1}{2} - f) = 2^{1-t} + \sum_{\substack{S \subset [t] \\ |S| \text{ even}, |S| \geq k }} 2^{1-t + |S|} T_{\Psi'|_S}(f),\]
where we have used the fact that $s(\Psi') = k-1$. 
Consider the set $\mathcal S$ of 
$$\{\Psi'|_S :  S \subset [t], |S| = k, \deg(\Psi'|_S)= k-2\},$$
where $\deg(\Psi'|_S)$ denotes the degree of freedom, or the dimension of the solution space, of the system $\Psi|_S$. 
Note that $\Psi \in \mathcal S$ so in particular $\mathcal S$ is nonempty, and furthermore that each $\tilde \Psi \in \mathcal S$ satisfies $s(\tilde \Psi) = k-1$. Invoke Lemma \ref{l:one-bigger} on the set of (distinct) forms in $\mathcal S$; let $f_0$ be the function that lemma produces and let $\hat \Psi$ be the linear pattern that lemma produces. Let $f_1$ be the function produced by Theorem \ref{thm:main-2xk-uncommon} for the pattern $\hat \Psi$ (so $T_{\hat \Psi}(f_1) < 0$). Let $f_2 := f_0^{\otimes A} \otimes f_1$. Recall then that $T_{\Psi'|_S}(f_2) = T_{\Psi'|_S}(f_0)^A T_{\Psi'|_S}(f_1)$ for all $S$. Then choose $A$ to be sufficiently large (and even if $T_{\hat \Psi}(f_0) < 0$) so that 
\[ \sum_{S \in \mathcal{S}} T_{\Psi'|S}(f_2) < 0.\]

Next, we claim that amongst all $S\subset [t]$ such that $|S|$ even and $|S| \geq k$, the set of $S$ for which $\deg(\Psi'|_S) + |S|$ is minimal is precisely $\mathcal S$. To this end we claim that $\deg(\Psi'|_S) \geq k-2$ since $s(\Psi') = k-1$. Indeed, if $\deg(\Psi'|_S) \leq k-3$, then restrict to these columns and we get a linear relation, so a contribution from something smaller than $k-1$). Thus $\deg(\Psi'|_S) + |S| \geq  |S| + k  - 2$, so the minimisers of $\deg(\Psi'|_S) + |S|$ must certainly come from $S$ with $|S| = k$. Furthermore, amongst these, $\codim \Psi'|_S \in \{1,2\}$, so $\deg(\Psi'|_S) + |S|$  is minimised by precisely those $S$ with $|S| = k$ and codimension 2; this is exactly the set $\mathcal{S}$, proving the claim.

Let $1_0 : \F_p^B \to \{0,1\}$ be the characteristic function of zero. Let $f_3 = (p^{-B}f_2)\otimes 1_0$. Then $T_{\Psi'|_S}(f_3) = p^{-B(|S| + D(\Psi'|_S))} T_{\Psi'|_S}(f_2)$. Setting $B$ large enough we can conclude. 
\end{proof}

\bibliographystyle{abbrv}

\bibliography{sid2}

\end{document}